\documentclass[10pt,a4paper]{article}
\usepackage{amsmath,amssymb,amsfonts,amsthm}
\usepackage[latin1]{inputenc}
\usepackage[english]{babel}
\usepackage{graphics}
\usepackage[mathscr]{eucal}

\usepackage{hyperref}

\usepackage[all,2cell]{xy} \UseAllTwocells 		

\newtheorem{theorem}{Theorem}[section]							
\newtheorem{proposition}[theorem]{Proposition}
\newtheorem{lemma}[theorem]{Lemma}
     
 \theoremstyle{definition}
\newtheorem{definition}[theorem]{Definition}
\newtheorem{remark}[theorem]{Remark} 
\newtheorem{remarks}[theorem]{Remarks}
 \newtheorem{example}[theorem]{Example} 
 
  \newtheorem{comments}[theorem]{Comments}
 \newtheorem{conjecture}[theorem]{Conjecture}

\newcommand{\bbc}{\mathbb{C}}

\newcommand{\A}{\mathcal{A}}\newcommand{\B}{\mathcal{B}}\newcommand{\C}{\mathcal{C}}
\newcommand{\D}{\mathcal{D}}
\newcommand{\E}{\mathcal{E}}\newcommand{\G}{\mathcal{G}}
\renewcommand{\H}{\mathcal{H}} 	
\newcommand{\I}{\mathcal{I}}

\newcommand{\U}{\mathcal{U}}

\newcommand{\X}{\mathcal{X}}

\newcommand{\Cs}{{\mathscr{C}}}
\newcommand{\Ds}{{\mathscr{D}}}
\newcommand{\Es}{{\mathscr{E}}}\newcommand{\Fs}{{\mathscr{F}}}\newcommand{\Gs}{{\mathscr{G}}}
\newcommand{\Hs}{{\mathscr{H}}}

\newcommand{\Ss}{{\mathscr{S}}}

\DeclareFontFamily{U}{rsfs}{\skewchar\font127 }
\DeclareFontShape{U}{rsfs}{m}{n}{%
   <5> <6> rsfs5
   <7> rsfs7
   <8> <9> <10> <10.95> <12> <14.4> <17.28> <20.74> <24.88> rsfs10
}{}
\DeclareSymbolFont{rsfs}{U}{rsfs}{m}{n}
\DeclareSymbolFontAlphabet{\scr}{rsfs}

\DeclareMathOperator{\Ob}{Ob}
\DeclareMathOperator{\id}{Id} 
\DeclareMathOperator{\Hom}{Hom}
\DeclareMathOperator{\Bis}{Bis}

\begin{document}

\title{Spectral C*-categories and Fell bundles with path-lifting}

\author{\normalsize  
Rachel A.D. Martins 
\\
\normalsize  \textit{Centro de An\'alise Matem\'atica e Sistemas Din\^amicos, Departamento de
Matem\'atica,}
\\
\normalsize \textit{Instituto Superior T\'ecnico, Universidade T\'ecnica de Lisboa,}
\\
\normalsize \textit{Av.~Rovisco Pais 1, 1049-001 Lisboa, Portugal}
 \thanks{Email: rmartins@math.ist.utl.pt. Research supported by Funda\c{c}$\tilde{\mathrm{a}}$o   para as Ci\^encias e a Tecnologia (FCT)
including programs POCI 2010/FEDER and SFRH/BPD/32331/2006.}}

\date{\normalsize{\today}}

\maketitle

\begin{abstract}
 Following Crane's suggestion that categorification should be of fundamental importance in quantising
gravity, we show that finite dimensional even $S^o$-real spectral triples over $\bbc$ are already
nothing more than full C*-categories together with a self-adjoint section of their domain maps,
while the latter are equivalent to unital saturated Fell bundles over
pair groupoids equipped with a path-lifting operator given by a normaliser. Interpretations can be made
in the direction of quantum Higgs gravity. These geometries are automatically quantum geometries and we
reconstruct the classical limit, that is, general relativity on a Riemannian spin manifold.
\end{abstract}


\section{Introduction}
Complementary to traditional studies towards quantum gravity, Connes' approach to unification is to
first geometrise the other three fundamental forces rather than to directly quantise the gravitational
(geometrical) force. Connes' and Chamseddine's spectral action principle \cite{sap} \cite{gravity} is a
classical theory of all four forces on a manifold consisting of a product of the conventional 4-space
and a finite noncommutative internal space\footnote{reminiscent of the Kaluza-Klein internal space}
including their theory of classical gravity (or isospectral \cite{sap} general relativity) on both
factors of the product space (the equivalence principle on the noncommutative internal space was studied
by Sch\"ucker, \cite{ncg and sm},\cite{forces}.) 

In a quantum (spectral) gravity \cite{essay} on the
discrete component, the Dirac operator would be an observable for gravity with eigenvalues the fermion
masses and the fundamental excitations would be holonomies in internal space. 

Following Crane's
suggestion 
that 
categorification is an important feature in quantisation and that a replacement of the traditional
analysis of the continuum might be replaced by topics in category theory \cite{category qg,qg,clock}, we
infer for those two reasons that a categorification program for noncommutative geometries might be
interesting. To this end, we explore here a connection between finite spectral triples, C*-categories
with an additional item of geometrical data and Fell bundles with path-lifting. Our inspiration comes
from the path integral approach, while Aastrup, Grimstrup and Nest \cite{AG1}, \cite{AG2}, \cite{AGN1},
\cite{AGN2} 
have already constructed a quantisation scheme for the product space using loop quantum gravity
techniques. 

There are  already many features of quantum mechanics in the noncommutative standard model:
the Dirac operator is an unbounded operator on a Hilbert space, the fermion sector of the action
resembles an expectation value, the spectral action is a topological invariant and resembles a path
integral\footnote{pointed out by Grimstrup \cite{AG1}} and as we see below, finite spectral triples have
a natural categorical representation.


In the formalism we construct, both the algebra representing a deformed phase space and the algebra of
functions on the configuration ``space'' are noncommutative and arise from the sectional algebra of a
Banach bundle. Some of the concepts are related to deformation quantisation theories especially the
tangent groupoid quantisation and we draw from geometrical notions inherited from the bundle structure.
In particular, we show that spectral C*-categories are equivalent to a certain class of Fell bundle with
a notion of path-lifting analogous to a path-lifting in a vector bundle.

Other points of view on categorification of spectral triples include  \cite{MarcolliCat,MeslandCat, BCL
mst, BCL Cncg}. In each of these cases, spectral triples are objects and morphisms are constructed
between them. The approach here is different but complementary: we focus on the internal categorical
structure of spectral triples and we present some first steps towards a quantisation program for Connes'
non-commutative general relativity. 

As well as being natural, viewing spectral triples as categories in
themselves (a kind of ``inner categorification'') might provide a tool for algebraic generalisations of
the path integral formulation of quantum gravity wherein the (space-time) manifold is the entire
category and the objects are just its boundaries, so that all the dynamics are formulated by morphisms
within one category. Nevertheless, other topics in categorification of spectral triples \cite{BCL
stm,BCL cqp} also have algebraic quantum gravity motivations and modular spectral triples and
categorical 
spectral triples have a strong intersection with our approach here: a spectral  C*-category aims to
interpret the spectrum of the Dirac operator in terms of the generating operator of a non-commutative
geodesic \cite{essay} and modular spectral triples \cite{BCL mst} and categorical spectral geometries
\cite{BCL Cncg} provide a very similar point of view. We aim to incorporate into their description the
structure of a (Banach bundle) C*-bundle reflecting the bundle structure of a Riemannian spin manifold. 

 \begin{definition}\cite{BCL Cncg}
  In short, a \emph{categorical spectral geometry} is a triple $(\Cs, \Hs, \Ds)$  given by a
pre-C*-algebra $\Cs$, a Hilbert C*-module $\Hs$, a (object bijective) functor $\rho : \Cs \to \B(\Hs)$
and a generator $\Ds$ of a unitary 1-parameter group on $\Hs$ respecting $\B(\Hs)$ with smoothness
condition: $[\Ds, \rho(x)]$ is extendable to a bounded operator on $\Hs$.
 \end{definition}

The definitions we introduce below of spectral C*-category and Fell bundle triple provide
generalisations and clarifications of Fell bundle geometries \cite{sc}, which are slightly more
elaborate containing the extra axioms (reality and Poincar\'e duality) needed to make them into real
spectral triples. Viewing spectral triples as Fell bundle geometries allowed us to make predictions
about the fermion mass matrix - up to the caveat that one does not have available all the technical
details for the final non-commutative standard model. Fell bundle geometries incorporate several
desirable features absent from real spectral triples including: a mass matrix approaching that of the
empirical fermion mass matrix, a bundle structure that reflects that of a Riemannian spin manifold, a
clue to a clearer meaning of orientability in non-commutative geometry and the availability of new
connections between physics and non-commutative geometry such as an application of the tangent groupoid
quantisation.

\subsection{Categorical switch-of-focus}

Isham, Crane and others explain that future progress in quantum gravity must involve mathematical
progress in the understanding of space-time on the smallest scales \cite{Isham,wqg}. Here we explain two
points of view on quantum space-times that underpin the ideas made precise in
this paper using algebraic techniques.

Although it is traditional to pass to a non-commutative generalisation of a topological space by
invoking the Gelfand-Naimark theorem and 
viewing non-commutative algebras as generalised algebras of functions, it is also true
that $C_0(X)$ is the algebra of sections vanishing at infinity of a line bundle over $X$, so there is an
argument to think of
non-commutative algebras as generalised algebras of sections, where \emph{the fuzzy points hide in the
fibres} over (a tangible) $X$ instead of in a fuzzy space ``$X$''. 

Let $(\Es,\pi,\G)$ be a Fell line bundle over a pair groupoid $\G=M \times M$ such that $\G$ is
interpreted as the
deformed tangent bundle over a simply connected compact manifold $M$. Each fibre $\bbc$ is associated 
to a point in $M$. Call the
enveloping algebra $C^*(\Es^0)$ of $(\Es^0,\pi,\G_0)$, the configuration algebra and call the algebra
associated to $\G$, the observable or the coordinate algebra $C^*(M \times M) = C^*(\Es)$. Switch
the focus away from the points in $M$ to the space of fibres of $\Es^0$, so that $\Es$ becomes a
generalised deformed tangent bundle. 

Now generalise the above example so that instead of $\bbc$ the fibres
of $\Es^0$ are simple matrix algebras over $\bbc$. Fibres are not necessarily isomorphic to eachother
but they are Morita equivalent. Now, \emph{$M$ has completely lost its interpretation as the
configuration space}, which is now a virtual space and is formalised only through the space of fibres
of $\Es^0$. A \emph{categorical switch-of-focus} (also appearing in \cite{sc}) $\Ss$ is a
category in which objects are given by the fibres over the objects of $\G$ whose morphisms are elements
of fibres over the morphims $\G_1$ of $\G$. Each object and each arbitrary disjoint union of objects
models a space-time region. \footnote{If $\E$ is saturated, then $\Ss$ is a ``weakened'' form
of groupoid (that is, a groupoid where the inverse and composition and unit axioms hold only up to
isomorphism).
$\Ob(\Ss)$ is a discrete Grothendieck site with morphisms given by unions of fibres. This categorical
structure is closely related to a quantum geometry \cite{wqg,mcqg} but these details may be explored
elsewhere.}

Secondly, in a context of space-time on the smallest scales, tools from traditional
calculus are no longer available (a calculation involving an infinitesimal distance or an
unbounded operator will only
give a physically meaningful answer in the much larger scale laboratorial measurable context) and there
is a need to
reformulate the geometrical data appearing in infinitesimal form (\cite{Isham}), such as a connection on
a vector bundle. 

The following
are three instances in which information of an infinitesimal type can be replaced by finite data through
integration: (i) Barrett and others have shown that given just the holonomy of a bundle, they
were able to reconstruct the entire bundle and connection from that data
(\cite{hol},\cite{axiomatic},\cite{group of
loops}). (ii) A choice of connection on a vector bundle over $M$ is equivalent to a choice of
representation of the fundamental groupoid $\Pi_1(M)$. (iii) The tangent bundle $TM$ over $M$ is
integrated to the $\Pi_1(M)$ (in this case $M$ is simply connected, so $\Pi_1(M)$ is the pair
groupoid over $M$) where infinitesimal data given by the tangent
vectors is replaced by a set of flows. So \emph{when an integration is performed, the following two 
things happen: a 
deformation or non-localisation of the geometry and a categorification of the geometry.}  This reflects
Ehresmann's point of view that geometry is the study of differentiable categories and it implies that
the following are almost synonyms: deforming to a non-commutative algebra, quantisation and inner
categorification. 

Below we demonstrate that a \emph{finite spectral triple is manifestly already a
category and we argue that this means that the non-commutative standard model is not completely
classical.}

\section{Preliminaries}

Because we will be drawing on material from several mathematical disciplines  (and arguing that certain
of their basic ideas can be unified through concepts from mathematical physics), readers may be familiar
with some topics and not others. Therefore we briefly introduce some of the ideas, recall the main
definitions and provide references. 

For an introduction to category theory the reader may wish to consult \cite{cwm}  or refer to several
introductory sources available online such as \cite{oosten}.

\begin{definition}\cite{GLR},\cite{Mitchener}
A \emph{$C^*$-category} is a category $\C$ in which for all objects $A,B\in\Ob_\C$, the homsets $\C_{AB} :=
\mathrm{Hom}_{\C}(B,A)$ are complex Banach spaces, the compositions are bilinear maps such that $\parallel xy
\parallel ~ \leq ~ \parallel x \parallel \cdot \parallel y \parallel$ for all $x \in \C_{AB}$, $y \in \C_{BC}$
and there is an involutive antilinear contravariant functor $* : \C \rightarrow \C$ preserving objects such
that $\parallel x^*x \parallel ~ = ~ \parallel x \parallel^2$ and such that $x^*x$ is a positive element of
the $C^*$-algebra $\C_{AA}$ for every $x \in \C_{BA}$ (i.e. $x^*x = y^*y$ for some $y \in \C_{AA}$).
\end{definition}

Note that each $\C_{AA}$ is a $C^*$-algebra and also that it possesses a unit element due to the identity
axiom in category theory. 

\begin{example}\cite{GLR}
The category of Hilbert spaces and bounded linear maps.   
\end{example}

\begin{example}\cite{GLR}
The category Rep($A$) of representations of a $C^*$-algebra $A$ on a Hilbert space and intertwining
operators. 
\end{example}

\subsection{Groupoids}

First we recall some notation conventions and basic facts. A groupoid $\G$ is a small category in which
all
morphisms or arrows $g$ are invertible and a group is a groupoid with just one object. The set of
objects $\Ob$, which is often referred to as unit space, is denoted $\G_0$ or $\G^0$. The domain (or
source) and range (or target) maps $(d,r):\G \to \G_0$ control the partial composition rule. To each
object in $G_0$ there belongs an identity or unit and the set of arrows $h$ satisfying  $h=g^*g$ can be
identified with the set of units, which due to the ``arrows-only'' picture of categories \cite{cwm}, is
identified with
$\G_0$. In this context, the set of arrows $\G_1$ is identified with $\G$. 
A groupoid may be equipped with a topology $\Omega(\G)$. A global (resp. local) bisection
of a topological groupoid is a section $x: \G_0 \to \G$ of $d$ such that $r \circ x :\G_0 \to \G_0$ is
a
(resp. partial) homeomorphism. Let $\I(\G)$ be the inverse semigroup of local (open) bisections of $\G$
and let $\Bis(\G)$ denote the group of global bisections. An \emph{\'etale} groupoid is a topological
groupoid in which $d$ is a local homeomorphism. An \'etale groupoid has a covering obtained from the
images
of open bisections and the topology generated by them is an involutive quantale $\Omega(\G)$ called an
inverse quantal frame. See \cite{Pedro} for a study of inverse quantal frames. An abstract inverse
semigroup can be can be represented by actions of partial or pseudo homeomorphisms of $M$, see
\cite{Dima} for a study of representations of inverse semigroups.


A smooth or Lie groupoid is a groupoid in which both $\G_0$ and $\G$ are manifolds, $d$ and $r$ are
submersions and all category operations are smooth. For $G$ a Lie
groupoid over a space $\G_0 = M$, an identification is often made of an arrow as a path (up to a certain
equivalence) of a point particle in a topological space and then a bisection can be used to describe
the set of paths
in $M$ transcribed by a system of particles. $\Pi_1(M)$ denotes the fundamental groupoid over $M$, which
is the set of homotopy equivalence classes of paths on $M$. In the case that
$M$ is simply connected, $\Pi_1(M) = M \times M$, the
pair groupoid over $M$.  A principal groupoid is one in which
$(r,d):\G \to \G_0 \times \G_0$ is injective, in other words it is an equivalence relation on $\G_0$.
The
pair groupoid $\G=M \times M$ is obviously principal as it is the maximal equivalence relation on
$\G_0$. The bisections of a pair groupoid are in one to one correspondence to diffeomorphisms of the
base.

\subsubsection*{The tangent groupoid quantisation}

Let $M$ be an $n$-dimensional manifold. Recall that the tangent bundle $TM$ and the cotangent bundle
$T^*M$ over $M$ are both Lie groupoids over $M$ locally diffeomorphic to $\mathbb{R}^n \times
\mathbb{R}^n$. As Lie algebroids, $TM$ and $T^*M$ both integrate to the groupoid $M \times M$
\cite{Lectures}. 

Here is a very brief overview of the idea of the tangent groupoid. For the
details see \cite{Connes' book}. This is a deformation type quantisation procedure through asymptotic
morphisms. It involves a particle
system on a manifold $M$. 
The cotangent space of $M$ captures the phase space and its $C^*$-completion of sections $C^*(T^*M)$ is
conventionally and classically the algebra of observables. Deformation quantisation programs replace it
with a non-commutative algebra. The tangent groupoid is given by
\begin{displaymath}
 \mathcal{G}M = TM \times \{  0 \}  \cup  M \times M \times (0,1].
\end{displaymath}
where $\hbar$ is a continuous parameter taking values in an interval of the real line $[0,1]$. The
classical limit is obtained as $\hbar$ `goes to zero'. The $C^*$-algebra of the tangent groupoid is the
algebra of continuous sections vanishing at infinity of the union of a norm continuous field of
$C^*$-algebras $A_{\hbar}$ over
the space of $\hbar$s. The asymptotic morphism is a morphism from the algebra $A_0$ over $\hbar = 0$
to any of those over $\hbar \neq 0$. At $\hbar=1$ we have the C*-completion $C^*(M \times M)$ of the
non-commutative convolution algebra of the pair groupoid over $M,$ whose elements extend to the compact
operators on a certain Hilbert space, $L^2(M)$. Again, for details see \cite{Connes' book}.


\subsection{Banach bundles, C*-bundles and Fell bundles}

Banach bundles are different from fibre bundles in that the absence of transition  functions allows them
to be non-locally trivial. Following J. Fell-R. Doran [\cite{Fell Doran} Section I.13] or N. Weaver
[\cite{Weaver}, Chapter 9.1] we have
the following definition of Banach bundle\footnote{Note that our norms are supposed to be continuous.}

\begin{definition}
A \emph{Banach bundle} ($\E, \pi, \X$) is a surjective continuous open map $\pi: \E \rightarrow \X$ such that
$\forall x \in \X$ the fibre 
$\E_x:=\pi^{-1}(x)$ 
is a complex Banach space and satisfies the following additional conditions:

\begin{itemize}
\item 
the 
operation 
of addition $+ : \E \times \E \rightarrow \E$ is continuous on the set 

$\E \times_{\X} \E := \{ (e_1,e_2) \in \E \times \E \ | \ \pi(e_1) = \pi(e_2) \}$,
\item 
the operation of multiplication by scalars: $\bbc \times \E \rightarrow \E$ is continuous,
\item 
the norm 
$\parallel \cdot \parallel : \E \rightarrow \mathbb{R}$ is continuous,
\item 
for all $x_0 \in \X$, the family $U_{x_0}^{\mathcal{O},\epsilon} = \{e \in \E \ | \ \parallel e \parallel <
\epsilon, \pi(e) \in \mathcal{O} \}$ where 
$\mathcal{O} \subset \X$ is an open set containing $x_0 \in \X$ and $\epsilon > 0$ is a fundamental system of
neighbourhoods of 
$0 \in \E_{x_0}$.    
\end{itemize}
For a \emph{Hilbert bundle} we require that for all $x \in \X$, the fibre $\E_x$ is a Hilbert space. 
\end{definition}
  
\begin{definition}
 A \emph{C*-bundle} is a Banach bundle $\E^0$ in which each fibre is  a C*-algebra and there is a
C*-completion $C^*(\E^0)$ of the algebra of compactly supported sections of $\E^0$.
\end{definition}

A Fell bundle over a topological groupoid (for which we follow \cite{fbg}) is a generalisation of a Fell bundle over a
group (\cite{Fell Doran}), and also a generalisation of a $C^*$-bundle over a topological space:-

\begin{definition} \cite{fbg}   \label{defining list}
A Banach bundle over a groupoid $p: \E\rightarrow \G$ is said to be a \textit{Fell bundle} if there
is a
continuous multiplication $\E^2 \rightarrow \E$, where
\begin{displaymath}
 \E^2 = \{(e_1,e_2)\in \E \times \E \ | \  (p(e_1),p(e_2)) \in \G^2\},
\end{displaymath}
(where $\G^2$ denotes the space of composable pairs of elements of $\G$) and an involution $e
\mapsto
e^{\ast}$ that satisfy the following axioms. 

\begin{enumerate}
\item 
$p(e_1e_2) = p(e_1)p(e_2) \quad \forall (e_1,e_2) \in \E^2$;
\item 
The induced map $\E_{g_1} \times \E_{g_2} \rightarrow  \E_{g_1 g_2}$,\quad 
$(e_1,e_2) \mapsto e_1e_2$ is bilinear $\forall (g_1, g_2)\in \G^2$;  \item
$(e_1e_2)e_3=e_1(e_2e_3)$ whenever the
multiplication is defined; 
\item $\parallel e_1e_2 \parallel \leq \parallel e_1 \parallel \cdot \parallel e_2 \parallel, \quad \forall
(e_1,e_2) \in \E^2$;
\item $p(e^{\ast})=p(e)^{\ast},\quad \forall e\in \E$;
\item The induced map $\E_{g} \rightarrow \E_{g^{\ast}}, \quad e \mapsto e^{\ast}$ is conjugate
linear for all $g \in \G$;
\item $e^{\ast \ast} = e, \quad \forall e \in  \E$;
\item $(e_1e_2)^{\ast} = e_2^{\ast} e_1^{\ast},\quad \forall (e_1,e_2) \in \E^2$;
\item $\parallel e^{\ast} e\parallel =\parallel e \parallel ^2, \quad \forall e \in \E$;
\item $\forall e \in \E, \ e^{\ast} e \geq 0$ as element of the C*-algebra $\E_{p(e^*e)}$. 
\end{enumerate}
\end{definition}

The restriction of a Fell bundle to $\G_0$ is a C*-bundle $\E^0$ and its algebra $ C^*(\E^0)$ is called
the diagonal algebra of the Fell bundle. 

A Fell bundle is said to be \textit{saturated} if 
$\E_{g_1}\cdot \E_{g_2}$ is total
in $\E_{g_1 g_2}$ for all $(g_1, g_2)~\in~\G^2$ from which it follows that the
fibres
are imprimitivity or
Morita equivalence bimodules. A \textit{unital Fell bundle} is one in which every $C^*$-algebra has an
identity element. A \emph{Fell line bundle} is a Fell bundle with fibre $\bbc$.

A saturated unital Fell bundle over a maximal equivalence relation $\G = M \times M$ or pair groupoid on
a topological space $M$ is equivalent to a full C*-category $\C$ fibred over the maximal equivalence
relation on its object space $\Ob(\C)$. (See \cite{BCL Imp}.)

\begin{example} \label{Morita}
Consider a saturated Fell bundle $\E$ over the pair groupoid on two objects $\G \ni g, g^*,
gg^*, g^*g$. Since the fibres are imprimitivity or Morita equivalence bimodules the
$C^*$-algebra of sections or `enveloping' algebra  of the  Fell bundle $C^*(\E)$ is a Morita equivalence
`linking algebra'.  For linking algebras and related results on Morita equivalence see
\cite{BrownGreenRieffel}, \cite{Joita}. The fibre space of $\E$ defines a \emph{Morita category} $\C$, a
 full small C*-category. The objects $\Ob(\C)$ are a Morita equivalence class of C*-algebras. This can
be thought of as a weakened form of a groupoid since $\E_{g} \otimes \E^*_{g} \cong \E_{g
g^{\ast}}$ and $\E_{g} \otimes \iota_{(\E_{g})} \cong \E_{g}$.

The object $C^*$-algebras are denoted $\E_{gg^{\ast}}$ and $\E_{g^{\ast}g}$ and
using
lower case for their elements, and element $e$ of the linking algebra is given by:
\begin{equation}
e = \left(   \begin{array}{cc}
e_{gg^{\ast}} &    e_{g}     
\\
e_{g^{\ast}}       &     e_{g^{\ast} g}
\end{array}  \right)
\end{equation}
\end{example}

\begin{example} 
The convolution algebra of the principal groupoid $\G = M \times M$ over a discrete
space $M$ consisting of $n$ points is given by the matrix algebra $M_n(\bbc)$. $M_n(\bbc)$ is the the
sectional algebra of the complex line bundle $\E$ over the groupoid $\G$ and is the linking algebra of
the Morita equivalence bimodules whose elements are the morphisms in the category $\E$.
\end{example}

The following definition and example is an extract from \cite{C*diag}.
\begin{definition}
 Suppose that $\A$ is a C*-subalgebra of a C*-algebra $\B$. An element $b \in \B$ is said to normalise $\A$ if 
 \begin{itemize}
  \item $b^* \A b \subset \A$ 
  \item $b \A b^* \subset \A$ 
 \end{itemize}
The collection of all such normalisers is denoted $N(\A)$. Evidently, $\A \subset N(\A)$; further, 
$N(\A)$ is closed under multiplication and taking adjoints. A normaliser, $b \in N(\A)$ is said to be
\emph{free} if $a^2=0$. The collection of free normalisers is denoted $N_f(\A)$.
\end{definition}

\begin{example}
 Let $\B = M_n(\bbc)$, the algebra of complex $n$ by $n$ matrices. Choose a   set of matrix units, $\{
e_{ij} : 1\leq ~ i,j ~ \leq n \}$ (one has $e_{ik} = e_{ij}e_{jk}$ and $e_{ij}^*=e_{ji}$), and let $\A$
denote the diagonal subalgebra (viz, $\A$ is spanned by the $e_{ii}$s). Then $a = \Sigma \lambda_{ij}
e_{ij}$ normalises $\A$ if and only if for each $i$, $\lambda_{ij} \neq 0$ for at most one $j$, and for
each $j$,  $\lambda_{ij} \neq 0$ for at most one $i$ (i.e., at most one entry is non-zero in each row
and column). If $i \neq j$, $e_{ij} \in N_f(\A)$.  
 Let $P : \B \to \A$ be given by:
 \begin{equation*}
  P(a) = \Sigma e_{ii} a e_{ii}.
 \end{equation*}
This defines a faithful conditional expectation for which:
\begin{equation*}
 \mathrm{ker} P = \mathrm{span} N_f(\A).
\end{equation*}
\end{example}

\subsection{Real spectral triples}

In short, a real spectral triple \cite{Connes' book} is a triple $(\A, \H, \D)$ where $\A$  is a
pre-C*-algebra with  a faithful representation on a Hilbert space $H$, the Dirac operator $\D$ is a
self-adjoint, unbounded operator on $\H$ with compact resolvent such that $[\D,a]$ is a bounded
operator for all $a \in \A$. 
$\H$ is a left $\A \otimes \A^{opp}$-module (finite projective) with a real structure $J$ and
$\mathcal{Z}_2$-grading $\chi$. A real structure on a spectral triple \cite{reality} is given by an
antiunitary operator $J$ on $H$ such that $J^2 = \pm 1$, $\D J= \pm J \D$, $[a,b^{opp}]=0$,
$[[\D,a],b^{opp}] = 0$ where $b^{opp} = Jb^*J^*$ for all $b \in \A$ is an element of $\A^{opp}$. Real
spectral triples are Connes's noncommutative generalisations of Riemannian spin manifolds, to which
purpose $\A$, $\H$, $\D$, $J$ and $\chi$ must satisfy a set of 7 axioms detailed in \cite{gravity} such
as Poincar\'e duality and orientability. These axioms were designed to be adapted over time.

Connes's reconstruction theorem establishes that commutative real spectral triples are
equivalent to Riemannian spin manifolds, so a commutative Riemannian spin manifold is just a
special case of the set of all not-necessarily-commutative Riemannian spin manifolds. 





\subsubsection*{Fluctuations of the metric}

In \cite{forces} Sch\"ucker explains that while Einstein derived general relativity from Riemannian
geometry, Connes extended this to noncommutative geometry and as a result the other three
fundamental forces emerged with the gauge and Higgs fields as fluctuations of the metric. The
`almost commutative' spectral triple of the noncommutative standard model includes a commutative
space-time factor and a finite noncommutative factor. The latter is reminiscent of Kaluza-Klein
internal space, but in this case it is 0-dimensional. Connes' encodes the metric data in the Dirac
operator \cite{gravity} and his procedure for unification starts by describing the diffeomorphisms
giving rise to the equivalence principle as the spinor lift of the automorphisms of the algebra
transforming the Dirac operator.  The arising space of fluctuated Dirac operators $\D^f$  (\cite{ncg
and sm}) defines the configuration space of the spectral action:-
\begin{equation} \label{fluctuations}
  \D^f = \sum_{\textrm{finite}} r_j L (\sigma_j) \D L(\sigma_j)^{-1}, ~~r_j \in \mathbb{R}, ~~ \sigma_j
\in \textrm{Aut}(\A)
\end{equation}
where $L$ is the double valued lift of the automorphism group to the spinors. For the calculation of
the spectral standard model action see \cite{sap}. The result is the general form of the Dirac
operator with arbitrary curvature and torsion:
\begin{equation} \label{general D}
 \D =  \sum_i c_i \Big(\frac{\partial}{\partial x}_i + \omega_i \Big)
\end{equation}
The  `almost commutative' algebra of the noncommutative standard model comprises two factors,
 $\A = C^{\infty}(M) \otimes \A_{\mathrm{finite}}$ and `fluctuating' the Dirac operator that probes
 $M$ in the finite  space algebra $\A_F$, Connes obtains the standard model gauge
fields. In this noncommutative case, this means replacing the spinor lift in the above formula with
the unitaries of $\A_{\mathrm{finite}}$. Finally, fluctuating the Dirac operator
$\D_{\mathrm{finite}}$ that probes the finite space in $\A_{\mathrm{finite}}$ gives rise to the Higgs
field. This gives the Higgs field an interpretation as a gravitational connection on an additional
`dimension'.

\section{Finite spectral triples}

The following definition is an interpretation for the finite complex case of an even spectral  triple
\cite{Connes' book} incorporating the fact (\cite{Harti}) that an inner derivation $\delta_{\D}: a
\mapsto
[\D,a]$ from a Banach algebra $\A$ maps into a module $\B$ over $\A$ ($a \in \A$, $[\D,a] \in \B$). 

\begin{definition}
 A \emph{finite even spectral triple} $(\A, \H, \D_f, \gamma)$ over $\bbc$ consists of a  finite
dimensional complex algebra $\A = \bigoplus_i M_{n_i}(\bbc)$, $i = \{ 1...p \}$ and an enveloping
algebra $\B = M_{n_m}(\bbc)$, both represented (faithfully unless stated otherwise) on a Hilbert space
$\H$ isomorphic to $\bbc^{n_m}$, $m = \sum_{i=1}^{p}$. The triple includes a self-adjoint operator
$\D_f$ giving an inner derivation from $\A$ into $\B$ such that $\D_f \gamma + \gamma \D_f = 0$ where
$\gamma$ is a $\mathbb{Z}/2$-grading operator satisfying $\gamma^*=\gamma$,  $\gamma^2=I$, and with $a
\gamma - \gamma a = 0 $ for all $a \in \A$.
 


\end{definition}
 
Observe that in examples, $p$ tends to be the number 4 since fermions  are predicated by chirality and
parity (they are always labelled either right or left and whether they are particles or
anti-particles). 

It is clear that the set of summands of $\A$ fall into a Morita equivalence class of non-commutative simple algebras.

The following real structure allowed Connes to define the analogue of a Riemannian  spin manifold as a
real spectral triple together with a set of 7 axioms (see \cite{gravity}).

\begin{definition} \cite{reality}
 A \emph{real} spectral triple is a spectral triple $(\A, \H, \D, \gamma, J)$  (defined in \cite{Connes'
book,gravity}) in which the Hilbert space carries a real structure, making it into an $\A-\A^{opp}$
bimodule, which is an anti-linear operator $J$ satisfying $\D J = J\D$, $J^2=1$, $J=J^*=J^{-1}$, $JaJ =
b$,
$[J, \gamma] =0$ 
 $\forall a$ with $b \in \A^{opp}$, 
\begin{displaymath}
   J \binom{ \psi_1}{\bar{\psi_2}}=\binom{\psi_2}{\bar{\psi_1}}
        \quad (\psi_1,\bar{\psi_2}) \in  \H= \mathcal{H}_1 \oplus \mathcal{H}_2
\end{displaymath}
where the bar indicates complex conjugation.
\end{definition}

\begin{definition}
Let $\H = \mathcal{H}_1 \oplus \mathcal{H}_2$ be a finite dimensional Hilbert space  with a further
chiral decomposition or $\mathbb{Z}/2$-grading $\mathcal{H}_1 = \mathcal{H}_L  \oplus \H_R$, 
$\mathcal{H}_2 = \H_{2L} \oplus \H_{2R}$.
 A \emph{finite real spectral triple} is a finite spectral triple as defined above  with a real
structure and satisfying the Poincar\'e duality condition: $\dim \H_R - \dim \H_L \neq 0$.
\end{definition}

\begin{remarks}
 \begin{enumerate}
  \item 
The form the chirality grading operator $\gamma$ takes depends on the KO-dimension (in physical examples
 this translates into for example whether the signature is Lorentzian or Euclidean). In the definition
we used above we have $J \gamma = \gamma J$ which corresponds to the Euclidean signature.  In \cite{lng}
and \cite{snm} for example the chirality operator is defined by a matrix solving the constraints
$\gamma^* = \gamma$, $\gamma^2=I$ and $J \gamma = \gamma J$ corresponding to the Lorentzian signature.
\item 
The equation $\D_f \gamma + \gamma \D_f = 0$ was included in Connes' original axioms\footnote{Connes'
original intention was for fluidity, not to set axioms in stone.} as a condition to ensure a reflection
of the orientability of a Riemannian spin manifold. However, there may be examples of
twisted Fell line bundles \cite{C*diag,Renault} that may give rise to non-orientable Fell bundle
triples (see later) that still satisfy the condition. A discussion on modifying the orientability axiom
is included in \cite{Christoph}.
 \end{enumerate}
\end{remarks}



The following condition was put in by hand to make  examples respect the phenomenology\footnote{this was
to avoid colour symmetry breaking which is unacceptable as for example it predicts that photons have
mass},

\begin{definition} \cite{reality}, \cite{gravity} 
 An $S^o$-real finite spectral triple is one in which the  Hilbert space of dimension $2l$ carries an
additional grading with operator $\epsilon$ with eigenvalues $(I_l,-I_l)$ where $[\mathcal{D},
\epsilon]=0$, $[J,\epsilon]_+=0$ (if the triple has a real structure), $\epsilon^{\ast} = \epsilon$,
$\epsilon^2=1$.
\end{definition}

\begin{lemma}
 The Dirac operator $\D_f$ for a real, $S^o$-real finite  spectral triple (up to Poincar\'e duality)
with $p=4$ \cite{gravity} \cite{sap} is given by the $n_m$ by $n_m$ matrix,

\begin{equation}        \label{ad hoc choice}
\D_f =
\left(   \begin{array}{cccc}
0         &         M^{\ast}    &   0        &   0\\
M         &          0          &   0        &   0\\
0         &          0          &   0        & M^T\\
0         &          0          &   \bar{M}  &  0
\end{array}  \right)
\end{equation}
where $T$ denotes transposition.\end{lemma}
 (This result is the same for either signature. In the Euclidean non-commutative  standard model is
associated with the fermion double problem. This was solved in \cite{lng} on the switch to the
Lorentzian signature.) 
(Caveat: Such a triple does not satisfy the Poincar\'e duality axiom because $M$ (mass matrix) is a 
square matrix. To rectify this, we only need to declare that $M$ include a column of zeroes and formally
delete one of basis vectors but this is a step that can always been done at the end of a calculation.)

\begin{proof}
 This is shown by applying each of the matrix conditions set out included in  the definitions above
involving the operators $\D$, $\gamma$ (for either signature) and $J$, \cite{smv}.
\end{proof}

\section{Spectral C*-categories}



The next definition is similar to categorical spectral geometries by Bertozzini, Conti and
Lewkeeratiyutkul that we quoted in the introduction and they are a first building block of a Fell bundle
geometry \cite{sc}. The main difference compared to categorical spectral geometries is that we are
presently only dealing with the bounded operators.

Recall that in a category $\C$, a section (or coretraction) $g$ of a morphism $f: A \to B$ is a right
inverse for it, $f \circ g = \id_B$. Let $\Fs:\C \to \C$ be an endofunctor and let $\Gs$ be a section
for $\Fs$. By this we mean that $\Fs \circ \Gs = \id_{\C}$. Let $r$ and $d$ denote the range and domain
map of $\C$. In analogy with a bisection of a groupoid, not least, the groupoid given by the cotangent
bundle over a manifold, we equip a full C*-category with an additional piece of data\footnote{this data
will attain a geometrical interpretation}:-

\begin{definition}
 A \emph{spectral C*-category} is a small full C*-category $\C$ equipped with a continuous global
self-adjoint section (or coretraction) $\sigma$ of its domain map such that $r \circ \sigma :
\Ob(\C) \to \Ob(\C)$ is bijective.
\end{definition}

Given a full C*-category $\C$, the set of all such continuous self-adjoint sections $\sigma$ of the
domain map will be denoted $N(\C)$.

\begin{remark}
 There is a groupoid isomorphism from the maximal equivalence relation $\G = M \times M$ defined on the
space of objects of $\C$ to the Picard groupoid of a full C*-category $\C$. (See \cite{BCL Imp}.)
$\sigma$ can be described as a set of elements of the imprimitivity bimodules arising from a bisection
of this Picard or ``weak'' groupoid.
\end{remark}

\begin{example}
 Recall from \cite{BCL Imp} that a full C*-category is a category of  Morita equivalence bimodules, so
the objects form a Morita equivalence class and the homsets constitute imprimitivity bimodules. Observe
that if $\C$ is a finite dimensional C*-category, that is, its object algebras are matrix algebras and
morphisms are given by elements of the linking algebra obtained from the Morita equivalence bimodule
structure of the category (see \cite{Rieffel Morita,Joita}). Define the \emph{diagonal} algebra of $\C$
to be $\A = \bigoplus_i A_i$ where each $A_i$ is an object of $\C$. Then $\sigma$ is defined by a
(self-adjoint) matrix in which there appears precisely one non-zero block in each row and column of
blocks. 
\end{example}

\begin{example}
 Non-commutative Fell bundle geometries defined in \cite{sc} are variants of spectral  C*-categories
with the additional axioms of reality structure and Poincar\'e duality built in. All we know from
spectral triple theory and the non-commutative standard model is that the Dirac operator takes the form
\ref{ad hoc choice} where the mass matrix $M$ is just a general matrix over $\bbc$. However, with the
non-commutative standard model's finite spectral triple viewed as a Fell bundle geometry, we can make
use of the additional mathematical structures available from the Morita category and the result is a
Dirac operator $\sigma$ that is beginning to attain a much closer resemblance to the fermion mass matrix
constructed from empirical data. For details see \cite{preprint}. Note also that although this has not
been studied in detail, this formalism does not automatically preclude Majorana masses, it merely
ensures that there is one non-zero block in each row and column of blocks in $\sigma$, which means that
a fermion is allowed either a Dirac mass or a Majorana mass but not both. On the
other hand, in \cite{unitaryinvariant} Connes constructs a complete invariant of Riemannian geometries
and then (quoting from his paper) he shows that his new invariant played the same role with respect to
the spectral invariant as the Cabibbo-Kobayashi-Maskawa (CKM) matrix in the Standard Model plays with
respect to the list of masses of the quarks.
 \end{example}

\begin{remark}
If $p$ and $q$ are operators on an infinite dimensional Hilbert space $\H$, then the Heisenberg
commutation relation $pq -qp =i \hbar$ is not satisfied whenever $p$ and $q$ are bounded operators. For
this
reason, an observable given by a function of $p$ and $q$ is always given by an unbounded self-adjoint
operator. It is often
considered unfortunate that one has to resort to the mathematical convenience of formally integrating
observable operators in order to obtain an algebra of bounded operators so that in particular, one may
use techniques from C*-algebras. On the other hand, from the point of view of a quantum space-time,
where the Hilbert space must be finite dimensional and where the geometry is of a non-local nature
(we cannot accommodate the notion of an infinitesimal in a mathematical framework modelling the smallest
things that exist \cite{Isham}), a bounded observable operator is contextual.
\end{remark}

\subsubsection*{The set $N(\C)$}

\begin{proposition}
 (a) Let $\C$ be a small full C*-category. The set $N(\C)$ is an  involutive monoid. Let $N_i(\C)$
denote the group of invertible elements of $N(\C)$. (b) Let $\C$ be a small full C*-category with $p$
objects given by simple finite dimensional complex algebras $A_i = M_{m_i}(\bbc)$, $i= \{ 1..p \}$, $n =
\sum_i^p m_i$. The set $N(\C)$ is equal to the normalising set $N(\A)$ where $\A$ is the ``diagonal
algebra'' $\A = \bigoplus_i M_{m_i}(\bbc)$. 
\end{proposition}
\begin{proof}
 (a) Recall from \cite{BCL Imp} that a full C*-category is a category of Morita equivalence bimodules,
so the objects form a Morita equivalence class and the homsets constitute imprimitivity bimodules. Note
that the object set $\Ob(\C)$ is closed under direct sums. Therefore multiplying elements $b, b^*,c$ in
$N(\C)$ arised from the composition in the Morita category, hence $b^*b \in \Ob(\C)$, $bc \in \C$ and so
$N(\C)$ must be an algebraically closed set and is also of course closed under taking adjoints since
$\C$ is an involutive category. The unit is just the section of $d$ obtained from the units belonging to
all the objects of $\C$.
 (b) Kumjian's definition of the normaliser of a C*-algebra $\A$ was quoted earlier in the 
preliminaries section. This situation closely resembles the finite dimensional example we also quoted
under the definition except that here the diagonal algebra $\bigoplus_i \bbc_i$ is replaced by a larger
matrix algebra. Now the proof is clear from the observation that a section of the domain $d$ of $\C$
such that $r \circ \sigma : \Ob(\C) \to \Ob(\C)$ is bijective, is
a matrix in which at most one matrix block is allowed to be non-zero in each row and in each column.
(Earlier examples and remarks help to clarify the situation.)
\end{proof}

\begin{lemma}
 There is a group homomorphism from the group of global  bisections $\Bis(\G)$ of the pair groupoid $\G
= \Ob(\C) \times Ob(\C)$ to the group of unitary $u$ (invertible partial isometries) normalisers
$N_u(\A)$ and a *-functor $\pi : \E \to \G$ such that $\pi \circ u = \id_{G}$ for all $u \in N_u(\A)$.
\end{lemma}

\begin{remark}
 Let $(\A, \H, \D, \gamma)$ be an even spectral triple. The  commutators $[\D,a]$ $\forall a \in \A$
generate an $\A$-bimodule $\Omega^1_{\D}$, providing the space of 1-forms over the spectral triple
\cite{Connes' book},\cite{Eli 1} 
 Note that any $\sigma \in N_i(\A)$ defines an inner derivation $\delta_{\sigma}$ from $\A$ into $\B$, 
and this group, for all $a \in \A$, generates the (``observable'') enveloping algebra $\B$. 
\end{remark}

\subsection{Finite spectral triple categorification}

With the following proposition we argue that a (finite) spectral triple is \emph{already} a category.

\begin{proposition}[Finite spectral triple categorification]
 Finite even $S^o$-real spectral triples over $\bbc$ with $p=4$ (as defined above) are spectral 
C*-categories (as defined above).
\end{proposition}

\begin{proof}
 The algebra $\A$ of a finite spectral triple is given by a direct sum of simple matrix  algebras over
$\bbc$. Let each simple algebra $A_i$ define an object in $\Ob(\C)$. The full C*-category $\C$ comes
from the Morita category of isomorphism classes of Morita equivalence bimodules over these objects
wherein the algebra $\B$ provides the linking algebra or ``enveloping algebra'' for $\C$. The Dirac
operator \ref{ad hoc choice} is a self-adjoint matrix in which there appears one non-zero block in each
row and column of blocks, which is exactly what ensures that it is a section (or coretraction) of the
domain map of $\C$.
 \end{proof}
 
\begin{remark}
Note that the $S^o$-reality condition was put in by hand for spectral triple physics  but here it loses
this sense of artificiality due to the categorical structures we are using\footnote{$S^o$-reality
excludes Majorana masses so we would need to modify these constructions to allow their inclusion.}. For
$p > 4$ we would need a modified version of $S^o$-reality. 
\end{remark}

\begin{remark}
 In other work we constructed an example of a finite real spectral triple from a finite  dimensional
non-commutative Fell bundle geometry adding structure such as Poincar\'e duality and reality and we used
this to make physical predictions about the hitherto ad hoc choice of the mass matrix $M$. (Bearing in
mind the fact that the full details of the algebra for a canonical non-commutative standard model is not
yet available.) \cite{sc,preprint}.
\end{remark}

The following could be investigated in order to make connections with other authors'  approaches towards
categorification in non-commutative geometry such as \cite{BCL cqp,BCL stm,MarcolliCat,MeslandCat}
wherein objects are spectral triples.
\begin{conjecture}
 The category of finite even $S^o$-real spectral triples over $\bbc$ and spectral  triple morphisms
\cite{BCL stm} is equivalent to the category of finite dimensional spectral C*-categories and
involutive functors.
\end{conjecture}

\section{Fell bundles and path-liftings}

\subsection{Notions of generalised connection and geodesics}

For maximal clarity, let us briefly recall some basic facts on the geometrical concepts that we will
be generalising to an algebraic context. 

Consider a vector bundle $(E^0,\pi,M)$ over a simply connected manifold
$M$ and let $GL(E^0)$ denote its general linear groupoid, that is, the set of all linear
isomorphisms between each pair of fibres: $GL(E^0) = \{ \alpha: E^0_x \to E^0_y ~~\vert~~ x,y \in M,
~\alpha~ \mathrm{an~~ isomorphism} \}$ together with the canonical composition of isomorphisms,
inverses and units $\iota_{x}:E^0_x \to E^0_x$. We also recalled earlier that the fundamental
groupoid $\Pi_1(M)$ of homotopy equivalence classes of paths on a simply connected $M$ is given by the
pair groupoid $\G=M \times M$. Earlier we also recall that $\G=M \times M$ can be thought of as the
deformed or integrated tangent bundle over $M$.

A representation of a groupoid on $E^0$ is a groupoid homomorphism $\rho: \G \to GL(E^0)$. A
choice of representation $\rho$ of the fundamental groupoid is equivalent to a choice of flat connection
$\omega$ on $E^0$. This is because the transition functions of $E^0$ force $E^0$ to be trivial and
because since paths or flows in $M$ that differ by a diffeomorphism have been identified in the homotopy
classes which give the arrows of $\Pi_1(M)$, (any curvature of $M$ has effectively been divided out).
The
arrows in $\G=M \times M$ can be put in one to one correspondence with the set of shortest paths (or geodesic
flows) between each pair of points in $M$. Each morphism $\rho_g \in \rho(\G)$ can be constructed by
integrating the connection along a member of the equivalence class of paths given by a groupoid arrow in
$\G$. We say that the connection $\omega$ is the generator of the representation $\rho$. An isomorphism
in $GL(E^0)$ is called a parallel transport 
if it belongs to the representation generated by $\omega$.  The connection is a field of
infinitesimals, a smooth assignment of
geometrical data to each point in $M$. Observe that a unique representation $\rho(\G)$ can be
constructed
from a choice of continuous bisection of $GL(\E^0)$ via the connection data and so it follows that
\emph{a choice of connection on $E^0$ is equivalent to a choice of bisection.} 
\\

To consider curved
connections one will require either (a) a finer Lie groupoid that will treat as two different arrows two
paths that are not diffeomorphic to eachother or \emph{alternatively} (b) a non-commutative
generalisation of $E^0$. For the
latter, one begins with a C*-bundle $\E^0$ because these are Banach bundles, which are defined without
reference to transition functions and so a C*-bundle over the groupoid $M \times M$ is not necessarily trivial.
Moreover, as illustrated in the final section, the curvature may be sourced from the non-commutativity
in the fibres.
\\

Next we define a notion of path-lifting for the context of a C*-bundle $\E^0$
in an ambient Fell bundle $\E$. A unitary path-lifting will give rise to a generalisation of a parallel
transport. We borrow some physics terminology to set up new definitions without making
immediate physical interpretations (the fibres of $\E^0$ may be non-commutative algebras, so we
are being careful not to use language that interchanges between states of the algebra and
states of the Hilbert space). Let $(\E, \pi, \G)$ be a unital saturated Fell bundle over a pair
groupoid $\G = M \times M$ over a simply connected compact manifold $M =\G_0$. Let $\E^0$ denote
the C*-bundle given by the restriction of $\E$ to $\G_0$. Let $\C$ be the C*-category associated to
$\E$. Let $\H$ be a
possibly infinite dimensional Hilbert space carrying a representation of $C^*(\E)$ and $C^*(\E^0)$. Let
each $\Psi \in \H$ be called a ``state'' of an ensemble of ``systems'' parametrised by
$\G^0$ 
(or equivalently, by the space of fibres of the ``diagonal'' bundle $\E^0$). Each system $\psi$ has
dimension equal to the dimension of 
the fibre labelling it or parametrising it. Let $\psi_x$ denote the system at $x \in M$ for a given state $\Psi$.
 

\begin{definition}
 A \emph{path-lifting} for $\psi_x \in \Psi$ is an assignment of an element $e$ of the fibre
$\E_{(x,y)}$ to a choice of path $(x,y) \in \G$ with $\psi'_y := e \psi_x$ in a new state $\Psi'$. By
consistency, such a path-lifting above fixes a path-lifting for $\psi_y \in \Psi$ (a second
system in the same state as $\psi_x$) over $(y,x)$ given by $e^* \in \E_{(y,x)}$. 
\end{definition}

\begin{definition}
 A \emph{path-lifting operator} $PL:\H \to \H$ on a Hilbert space $\H$ is a section of $(\E,\pi,\G)$ 
supported on a global bisection of $\G$, which has been extended to a bounded operator on $\H$, thus
providing a path-lifting for every $\psi \in \Psi$ for each $\Psi \in \H$.
\end{definition}


\begin{lemma}
 Let $(\E,\pi,\G)$ be a Fell bundle as above whose
restriction to $\G_0$ is a trivial C*-bundle $\E^0$ with underlying vector bundle structure
$E^0$. A choice of unitary path-lifting section is equivalent to a choice of a flat connection on the
vector bundle $E^0$. 
 \end{lemma}
\begin{proof}
 This follows from the above discussions including because the unitary path-lifting section gives rise
to a bisection of $GL(E^0)$. 
\end{proof}

 \begin{remarks}
  \begin{enumerate}
   \item 
In view of Renault's paper (\cite{Renault}, proposition 4.7), we can confirm that the normalisers of
$C^*(\E^0)$ in $C^*(\E)$ are exactly the sections of $\pi$ supported on the local bisections of $\G$ and
so the group
of invertible path-lifting operators can be identified with the group of invertible normalisers
$N_i(C^*(\E^0))$ of $\A = C^*(\E^0)$. Since the unitary normalisers $N_u(\A) \cong \Bis(GL(\E^0))$, a
choice of group homomorphism from $\Bis(\G)$ to $N_u(\A)$ corresponds to a choice of flat connection on
the underlying vector bundle structure of $\E^0$ and it extends to a representation $\rho$ of $\Bis(\G)$
on $\H$.
\item
The inverse semigroup homomorphism $\I(\G) \to \I(GL(\E^0))$ extends to a representation
on $\H$. The non-zero entries of normalisers of $N(\A)$ that are not sums of other normalisers are
elements of $\E$. A unitary operator is an invertible partial isometry $v$. Since we are working with
C*-bundles which do not necessarily have isomorphic fibres, one should generalise the general linear
groupoid $GL(\E^0)$ to an inverse semigroup implemented by partial isometries (that is, $v$ is an
element of $\E$ with a unique quasi inverse $v^*$ satisfying $vv^*v=v$ and $v^*vv^*=v^*$).
  \end{enumerate}
 \end{remarks}



Let $M$ be a simply connected manifold. In view of the discussion above, we can label each arrow in the
deformed or integrated tangent bundle $\G = M \times M$ by a geodesic flow. In view of the categorical
switch-of-focus (see introduction), where we replace geometrical notions in the base space with
algebraic notions in the space of fibres, we now give $\E$ the interpretation of a generalised tangent
bundle (in the case that $\E$ is a line bundle, $C^*(\E) = C^*(M \times M)$) and to each unitary
path-lifting, we give the interpretation of a geodesic flow. However, to define a non-commutative
geodesic (to make precise the transport of fuzzy points in one fibre to fuzzy points in the next fibre)
this does not go quite far enough because in the case that
the fibres of $\E^0$ are not isomorphic (such as in a C*-bundle with non-commutative fibres over a
discrete
space or a Banach bundle that is not locally trivial) or even in any case when $\E^0$ is not identified
with its underlying vector bundle $E^0$, then $GL(E^0)$ no longer provides a good
description of parallel tranport. Since the unitary operators are the invertible partial isometries,
$GL(E^0)$ should be generalised in further work by dropping the condition of invertibility of its
morphisms, so that isomorphisms between fibres are replaced with linear *-homomorphisms
implemented by partial isometries $v$.   

\begin{definition}
 Let $(E,\pi,\G)$ be a unital saturated Fell bundle with not necessarily isomorphic
fibres, a \emph{non-commutative geodesic} is given by a path-lifting $e$ such that $e$ is a partial
isometry.
\end{definition}

The non-local (not relying on geometrical data of an infinitesimal nature and not relying on an
unbounded operator) object we have to model a connection, that is, a generator of a geodesic flow, is a
path-lifting section. This is why unitarity is not required in the definition of a path-lifting
introduced above and this is why we defined path-liftings first instead of generalised parallel
transports. Note also that a path-lifting operator will always be bounded. 

\subsection{Fell bundle triples}

A Fell bundle triple is the following algebraic generalisation of a geometrical space wherein the
``manifold'' is the space of fibres of a C*-bundle $\E^0$ and where the basic geometrical data is given
in the form of a path-lifting operator, instead of a metric or a connection or a curvature 2-form.
Below we show that Fell bundle triples are equivalent to spectral C*-categories, which in turn we have
already compared with spectral triples, the definitive non-commutative spaces. 

\begin{definition}
 A \emph{Fell bundle triple} $(\E, \H, PL)$ consists of a unital saturated orientable Fell bundle
$(\E,\pi,\G)$ over a pair groupoid $\G$ over a simply connected compact manifold $M$, a Hilbert space
carrying (faithful unless stated otherwise) a representation of the algebra $C^*(\E)$ with a
subrepresentation of $C^*(\E^0)$, together with a path-lifting operator $PL:\H \to \H$.
\end{definition}

\begin{proposition}
 Fell bundle triples and spectral C*-categories are equivalent.
\end{proposition}

\begin{proof}
 We have already recalled that saturated unital Fell bundles over pair groupoids are equivalent  to full
C*-categories. All that remains is to check that their geometrical data coincide. Let $A,B$ be objects
in $\Ob(\C)$ and let $x$ be the corresponding unit in $\G^0$. A \emph{path-lifting} for $\psi_x$ can
equivalently be viewed as a choice of target pair $(y,B)$ and an assignment of a morphism in $\Hom(A,B)$
to the groupoid arrow (or path) $(x,y)$ such that $\psi_y' = e \psi_x$ for a system $\psi'_y$ in a new
state. This fixes a path-lifting for $\psi_y \in \Psi$ (a second system in the same state as $\psi_x$)
over $(y,x)$ given by $e^* \in \Hom(y,x)$. It follows that the fact that a path-lifting operator $PL:\H
\to \H$ is defined from a section supported on a bisection of $\G$ means that it must be a coretraction
of the domain map of $C$; and so a path-lifting operator is a *-functor from $\G$ to $\C$ specifying an
element of a *-functor from $\Bis(\G)$ into $\C$ (or specifying an element of the group $N_i(\A)$) and a
continuous section $\sigma$ of $E$. Lastly, such a path-lifting satisfies $\sigma 
= \sigma^*$ because for every path from $x$ to $y$, there is a geometrically  equivalent reflected path
from $y$ to $x$, to which belong a common piece of geometrical data defined during any given event.
\end{proof}

We make the following new geometrical interpretation of finite Dirac operators:

\begin{proposition}
 Finite even $S^o$-real spectral triples over $\bbc$ are  Fell bundle triples wherein a Dirac operator
$\D_f$ is equivalent to a path-lifting operator $PL$. 
 \end{proposition}

\begin{example}
The diagram illustrates the support of $\D_f$ as in equation \eqref{ad hoc choice}.
 
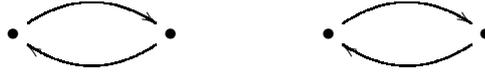
\vspace{0.5cm}
\begin{figure}[ht]
\qquad \qquad \qquad \qquad \quad
\xymatrix@R=4pc@C=4pc{
  {\bullet}  \ar @ /^1pc/ [r]    &   {\bullet}  \ar@ /^1pc/  [l]      &   {\bullet}  \ar @ /^1pc/ [r]  
&   {\bullet} \ar@ /^1pc/  [l] }
 \end{figure}
\vspace{0.5cm}

\begin{remark}
We remark that a matrix $M$ satisfying the equation of motion $M(M^*M-I)=0$ (see \cite{smv}) is a 
partial isometry and attains the interpretation of a non-commutative geodesic. The mass matrix itself on
the other hand, attains the interpretation of the generator of the geodesic.  
\end{remark}
\end{example}

\begin{example}
 A categorical spectral geometry $(\Cs, \Hs, \Ds)$ \cite{BCL Cncg} (quoted in the introduction) is
essentially data-equivalent to a spectral C*-category and a Fell bundle triple although there are
small technical differences such as $\Cs$ is a pre-C*-category. In general,  (and if $\Ob(\Cs)$
has an underlying structure of a vector bundle) the generator $\Ds$ is extended to an unbounded operator
on $\Hs$ and can be constructed from the geometrical data in $PL$ where $\Ds$ is interpreted as the
connection on $\Ob(\Cs)$. Modular spectral triples \cite{BCL mst} are similar constructions to
categorical spectral triples and involve more details arising from the dynamical quality of von Neumann
algebras.
\end{example}

\begin{example} \label{key example}
 Before we construct the following example of a Fell bundle triple, first consider an example of a
semidirect project Fell bundle $\G \ltimes C^*(\E^0)$ similar to an example from \cite{fbg}. Let $\E$ be
a saturated unital Fell bundle over a topological groupoid $\G$. 
The product of elements
$e_1 = (g,a)$ and $e_2 =(h,b)$, $a,b \in C^*(\E^0)$, for each pair $(g,h)$ such that $gh = h \circ g \in
\G$, $d(g) = \pi(a)$, $r(g)=\pi(b)$, in the Fell bundle is given by: 
 
 \begin{equation*}
  e_1e_2 = (gh, \alpha_g(a)b)
 \end{equation*}
where $\alpha_g(a) = u a u^*$ wherein $\alpha_g$ is a linear *-isomorphism of fibres (element of
$GL(\E^0)$) and $u$ is a unitary free normaliser over $g$ varying continuously over all $g \in \G$,
$\pi(u) = g$, $\pi(u^*) = g^*$. $e_1^* = (g^*,a^*)$.

To construct an example of a Fell bundle triple, we set $\E$ to be a unital Fell bundle over $\G$ the
pair groupoid over a compact simply connected space. In the cases that $\E$ is also a vector bundle (as
well as a Banach bundle) then $\E$ is a trivial bundle and so orientable. A path-lifting operator $PL$
is obtained from the unitary elements defined by the field $\alpha_g$ for each $g$ of a global bisection
$x$ of $\G$. $PL$ will automatically be self-adjoint since $u_{g^*} = u^*_g$ for all $g$ in $x$. 
\end{example}


\begin{comments}
\begin{enumerate}
 \item The algebra $\A = C^*(\E^0)$ we call the configuration algebra and the algebra $C^*(\E)$ we call
the
observable or coordinate algebra. Note that if $\E$ is a line bundle then $C^*(\E)=C^*(M \times
M)$\footnote{We take $C^*(M \times M)$ to be a C*-completion of the convolution algebra of $\G$ and
identify it with a C*-completion such as $C^*_{red}(\G,\pi)$ of the algebra of compactly supported
sections of the line bundle over $\G$.}, which is the observable algebra in the tangent groupoid
quantisation. (As already mentioned, traditionally, unbounded observable operators are exponentiated to
obtain a bounded
(unitary) operator and element of a C*-algebra of observables.) 
\item Other work has involved defining an inner automorphism of a C*-bundle $\E^0$ as
an inner automorphism of the algebra commuting with the Banach bundle projection map $\pi$ and defining
a representation of $\Bis(\G)$ as a group homomorphism into this inner automorphism group \cite{DS}.
The context of \cite{DS} treats the relationship
between diffeomorphisms of $M$ and the inner automorphism group of the C*-bundle to illustrate and study
Connes' algebraic analogy of Einstein's equivalence principle. 
\item In further work we hope to consider
examples where the C*-completion of the algebra of sections of the C*-bundle is a von Neumann algebra
with with weight $w$ such that $w$ corresponds to the generator of the inner automorphisms of the
C*-bundle in order to give $w$ an interpretation as generalised 
connection on a C*-bundle. We expect this to involve modular spectral triples \cite{BCL mst}. (A
quantisation program should involve a characterisation of the states of the system and this context
might lead to a clearer description of geometrical states for some form of algebraic quantum gravity.)
\end{enumerate}
\end{comments}

\section{The classical limit}

Recall that a \emph{spin structure} on an orientable Riemannian manifold $M$ is an
equivariant lift of the oriented orthonormal frame bundle respect to the double covering
$P : \mathrm{Spin}(n) \to \mathrm{SO}(n)$. In a reconstruction theorem, the non-commutative geometor
checks that any Riemannian spin  manifold can
be constructed in full (or up to a possible torsion term) from only the data in a real spectral triple.
The material in this paper is towards a quantisation program, so constructing a Riemannian spin manifold
from an example of a spectral C*-category or a Fell bundle triple is about finding the classical
limit, not a reconstruction theorem. For a recent paper on reconstruction theorems see
\cite{Branimir}. 
\\

We have already argued that spectral triples have a natural categorical aspect coming
from the associated Morita category. Moreover, we have seen that \emph{viewing every spectral triple as
a spectral C*-category automatically incorporates a quantisation aspect in the sense that the
non-commutative algebra of observables $C^*(\E)$ is pre-built-in to the formalism, without any need to
actively deform}. This confirms Crane's suggestion that in quantum gravity one 
should think of the whole system in terms of categorical  structures. Our examples with
non-commutative configuration algebras do not even have classical limits as $\hbar$ tends
to zero because the obstruction to the observable algebra tending to become commutative exists in the
fact that the configuration algebra itself is non-commutative. (The limit of these geometries in the
sense of a continuum limit will be in the sense that as one takes a larger scale view, the fuzzy
points in each fibre in $\E^0$ merge to a point as the non-commutative fibres $M_n(\bbc)$ are replaced
by $\bbc$. As the points move closer and closer together, we can go back to using the continuum to
describe the space together with the traditional calculus.) 
\\

This formalism may help towards solving the
problem that the non-commutative standard model is traditionally a classical (general relativity)
theory. Another open problem is the difficulty of deformation quantisation on curved space-times, which
we are considering by allowing the configuration algebra (not only the observable or coordinate algebra)
to be non-commutative. To do quantum gravity on a Fell bundle, first we need to do general relativity on
a Fell bundle. In the
preliminaries, we briefly explained and gave references to Connes' non-commutative general relativity
with regard to fluctuations of the Dirac operator. We should reconstruct general relativity on a
Riemannian spin manifold from the geometrical data provided by a spectral C*-category and its Fell
bundle.

Within the context of the categorical switch-of-focus (see introduction),  we have discussed in the
previous section how a parallel transport can be thought of as a non-local connection and mentioned the
fact that given a holonomy group, a unique bundle and connection can
be fully constructed. (So the data given by a holonomy group already tells us all we need to know about
the curvature of a manifold even without having access to any infinitesimal data.) Analogously, the
Dirac operator also carries geometrical data of an infinitesimal nature about a Riemannian manifold. If
the operator $\sigma$ is a (bounded) generalised Dirac operator as claimed, then there will exist a
non-commutative geometry given by some spectral C*-category and Fell bundle triple over a Riemannian
spin manifold $M$ whose classical limit will provide a Riemannian spin manifold complete with its
traditional Dirac operator, an unbounded operator on the Hilbert space of square integrable sections of
the spinor bundle satisfying Connes' spectral triple axioms. 
\\

As a bundle of C*-algebras, a Clifford bundle has an equivalent description as a C*-bundle  and is
denoted $Cl^0$ or $\E^0$ below. We begin with the trivial case of a Riemannian manifold $M$ together
with a Clifford algebra given by $\bbc$ and where the Riemannian manifold is flat. We have been and we
continue to restrict to compact manifolds due the presence of the unit in the axioms for a category.

\begin{proposition}
 Let $(\E,\pi,\G)$ be a Fell line bundle over the pair groupoid $\G$ over a simply connected compact
orientable Riemannian manifold $M$, with diagonal C*-bundle $\E^0$ and diagonal (or configuration)
algebra $\A=C^*(\E^0)$ and
with enveloping (or coordinate) algebra $C^*(\E) = C^*(M \times M)$\footnote{This is a C*-completion
$C^*_{red}(\G)$ of the algebra of compactly supported sections on $\G$.}. Let $\C$ be the associated
C*-category. Now introduce geometrical data given by a choice of $\sigma$ and $PL$ as defined above. We
can reconstruct from this non-commutative geometry, the classical limit, which is a trivial (flat)
Riemannian spin manifold with Clifford fibre $\bbc$ with a Dirac operator in a coordinate frame in which
the (flat) connection is zero by choosing $\alpha=1$. That is, the Dirac operator is given by the
physicist's flat Dirac operator $\gamma^{\mu}\partial_{\mu} + 0$.
\end{proposition}

\begin{proof}
 Any Fell line bundle $\E$ can be constructed as a semidirect product bundle (due
to it being saturated and locally trivial) and so the context is described by the trivial case of
example \ref{key example}. This is because any Fell line bundle is saturated and for any saturated Fell
bundle, $C^*(\E^0)$ is regular in $C^*(\E)$ and the ``regular'' Fell bundles are the semidirect product
bundles. (See for example the introduction of \cite{Renault}). And so any element of $\E$ can be
expressed as $(g,a)$ with $\pi(a) = g \in G$ and the multiplication in the Fell bundle is given as in
\ref{key example} with $\alpha$ being the identity for all $g$ due to the triviality. To choose
$\sigma$, we choose a bisection $x$ of $\G$ and a general element $(g,a)$ of the fibre of $\E$
over each groupoid arrow $g \in G$ in the bisection $x$ supporting $\sigma$, with $\alpha$ defined by
$\alpha(a) = a$ for all elements $a \in \E^0$. The different possible choices of $\alpha$ correspond 
to the different choices of coordinate frames and flat connections. Let $(\A,\H,\D)$ be the 
spectral
triple associated to the Riemannian manifold. The square integrable sections $\H$ of the Clifford
bundle over $M$ carries a representation of the algebras $\A$ and $C^*(\E)$ and we identify $\H$ with
the
Fell bundle triple Hilbert space $\H$. All other structures pertaining to the Hilbert space can be
copied over from $\H$. 
 
 Now we use a technique from the tangent groupoid quantisation  (see \cite{Connes' book} and refer to
the brief description above in the preliminaries). Recall that the Lie algebroid $TM$ integrates to $M
\times M$ and so deriving the classical limit involves the reverse process of this integration.
Underlying the tangent groupoid quantisation \cite{Connes' book}, \cite{Paterson} is the fact that the
limit of a quotient is a derivative (consider a sequence of groupoid arrows $g_n$ where the arrows
become shorter and shorter until they have no length at all and become tangent vectors). Briefly, all
Cauchy sequences in $\frac{M \times M}{\hbar}$ converge in $\frac{TM}{\{0\}}$: consider a sequence, 

\begin{equation*}
 p_n ~ \rightarrow ~ q_n ~ \rightarrow ~ p ~~ \mathrm{with} ~~ (p_n, q_n) \in M \times M, ~~ p \in M,
\end{equation*}
then as 

\begin{displaymath}
  \hbar ~~ \mathrm{goes ~~to~~ zero,} ~ \frac{p_n - q_n}{\hbar} ~ \rightarrow ~ v ~~ \mathrm{where} ~~ v \in TM
\end{displaymath}
and we write

\begin{displaymath}
 (p_n, q_n, \hbar) ~ \rightarrow ~  (p, v, 0), ~~ \mathrm{or} ~~ (g_n, \hbar) ~ \rightarrow ~ (v,0) ~~
\mathrm{if} ~ g_n = (p_n,q_n)
\end{displaymath}

So at each (space-time) point $d(g)$ in $M$, the classical limit of $(g,a)$  is a tangent vector $v$
paired with (or contracted over space-time indices with) an element $a$ of the Clifford algebra as a
fibre over $d(g)$. 

\begin{equation*}
 D = \gamma^{\mu} \partial_{\mu}.
\end{equation*}
We do this at every point, that is, extending continuously over the space $M$.  
(To explain the reference to a direct multiplication involving the space-time index
contraction, first consider a finite dimensional example where $g$ as in $(g,a)$ can be considered as a
matrix unit $e_{ij}$ with $d(e_{ij})=i$ and $r(e_{ij})=j$. Now let $\G^0$ be a continuous space instead.
There is a fibre of $\E^0$ assigned to each point in the space-time and therefore each element $a$ comes
with space-time coordinates. So when $g$ goes to $v$, the vector $v$ is contracted with $a$ over the
space-time indices.)

Another view on the above construction comes from the fact that the algebroid of derivations $Der(Cl^0)$
on the bundle $Cl^0$ integrates to $GL(Cl^0)$. And therefore applying the above procedure to a sequence
in $GL(Cl^0)$, instead of an element of $TM$, the limit is an element of the Lie algebroid which we
identify
with a flat Dirac operator $D_{RSM}$ as above. With the identification of 
$Der(Cl^0)$ with $\mathrm{End}(Cl^0) \otimes  TM$, the Atiyah (short exact) sequence illustrates  the
point that $Der(Cl^0)$ is the extension of $TM$ by the spin group, which we identify with ker$\rho$,
where $\rho$ is the surjective map $\mathrm{End} (Cl^0) \otimes TM ~ \longrightarrow ~ TM$:

\begin{equation}
 0 ~ \longrightarrow ~ \mathrm{ker}\rho ~ \longrightarrow ~ \mathrm{End} (Cl^0) \otimes  TM ~
\longrightarrow ~ TM ~ \longrightarrow ~ 0
\end{equation}
\end{proof}

Next we consider curved Riemannin spin manifolds. 

Normally in order to incorporate curvature into a (simply) connected manifold,  we would need to replace
the (pair) fundamental groupoid with a groupoid in which the diffeomorphisms are not divided out, as
that would be the only way to achieve curvature in view of Einstein's equivalence principle. However,
Connes has provided us with an alternative approach, and so we ``fluctuate'' (see preliminaries) the
flat or initial Dirac operator in the non-commutative fibre of the Clifford bundle to reconstruct a
general Dirac operator on a Riemannian manifold $M$. Examples of Clifford algebras include
$M_n(\bbc)$, $M_n(\mathbb{R})$, $M_n(\mathbb{H})$, $M_n(\bbc) \oplus M_n(\bbc)$, $M_n(\mathbb{R}) \oplus
M_n(\mathbb{R})$, $M_n(\mathbb{H}) \oplus M_n(\mathbb{H})$ and algebras isomorphic to them. Note that at
present the scope is limited to those Clifford bundles that can be equivalently described as locally
trivial C*-bundles with Clifford algebra over $\bbc$. Of course, it might be possible to extend the
formalism 
to include the Clifford algebras over $\mathbb{R})$ whose cooefficients are  in $\mathbb{R}$ or
$\mathbb{H}$ by generalising to ``real C*-bundles'' (i.e. Banach bundles fibred by real C*-algebras).
Nevertheless, the Clifford algebra $\mathbb{H} \oplus \mathbb{H}$ is already included up to isomorphism
since $M_2(\bbc) \cong \mathbb{H} \oplus \mathbb{H}$.

\begin{proposition}
 Let $M$ be a Riemannian spin manifold with Clifford bundle and C*-bundle denoted $Cl^0$ or $\E^0$.
There is a Fell bundle triple $(\E,\pi,\G,PL)$ or equivalently a spectral C*-category $(\C, \sigma)$,
such that the classical limiting geometry as $M \times M \to TM$ results in the original Riemannian spin
manifold.
\end{proposition}

The context here is the same as that of the previous theorem except that $\E^0$ is a  more general
C*-bundle and $PL$ is now a path-lifting corresponding to a curved connection. 

\begin{proof}
 Einstein's equivalence principle for general relativity involves selecting an initial flat metric  and
then fluctuating it, producing a curved metric. In non-commutative geometry, these diffeomorphisms
are realised as coordinate algebra automorphisms, $\phi$. A `fluctuation' of the  Dirac operator is
given by Connes as $L(\phi) \D L(\phi)^{-1}$ where $L$ is the spinor lift (a mapping of the 
automorphism group to its double cover). The image group of $L$ is the spin group, which is the inner
automorphism group of the Clifford algebra. We define a fluctuation of $\sigma$ to be (locally)
$\alpha(\sigma) = u \sigma u^*$ with $\alpha$ an inner automorphism of the Clifford algebra (locally)
$\alpha(a)=uau^*$. 

($\alpha$ induces isomorphisms of the fibres of the underlying vector bundle of $Cl^0$ and a
fluctuation induces a groupoid homomorphism, 

\begin{displaymath}
 \alpha \circ \rho : \G \rightarrow GL(Cl^0)
\end{displaymath}

 \begin{displaymath}
  g \mapsto \alpha((g,u)) = (g, \alpha(u))
 \end{displaymath}
since the whole groupoid representation can be constructed from a unitary section supported on a
bisection as explained earlier. This carries essentially a concept as described in \cite{IS}
in
spectral triple gravity. Note that if $C^*(Cl^0)$ is commutative then $\alpha$ is the identity,
implying
a conceptual link between curvuture and non-commutativity). In general non-commutative geometrical
contexts, we have Connes's fluctuations formula, quoted in the preliminaries (\ref{fluctuations}), 
which is a linear combination of fluctuations or diffeomorphisms and results in a curved Dirac operator
$\D^f$.  
The fluctuated Fell bundle Dirac operator is given (locally) by the formula:

\begin{equation}    \label{Cl}
         \D^f =   \sum_j r_j U_j (g, a) U_j^*;~~ U_j \in Cl(d(g)), ~ r_j \in \mathbb{R}
\end{equation}
the $U_j$s can be identified with elements of the spinor group because this is the structure group of
the Clifford bundle.

Carrying out the equivalence principle using Connes's method of taking linear combinations of
fluctuations beginning with an initial flat local Dirac operator $\D_{\mathrm{initial}} = (g,a)$, or
extended over $M$, $\D_{\mathrm{initial}} = \sigma$ and invoking the same argument as in the flat
geometry case about producing the classical limit of the Dirac operator we obtain,
\begin{equation*}
 \sum_j r_j U_j \D_{\mathrm{initial}} U_j^* ~ = ~ \sum_j  r_j U_j \sigma U_j^* 
 \end{equation*}
 and as $\hbar  ~  \rightarrow ~ 0$, 
 \begin{equation*}
  \sum_j r_j  U_j \Big(  \sum_i c_i \frac{\partial}{\partial x}_i \Big) U_j^*  ~ = ~  \sum_i c_i
\Big(\frac{\partial}{\partial x}_i + \omega_i \Big)
\end{equation*}

The result is a general Dirac operator on a Riemannian  spin manifold with arbitrary torsion and spin
connection $\omega$. The last line comes from Connes' non-commutative geometry, see for example
\cite{gravity}, and is a consequence of the fact that if $U$ is unitary then $U\D U^* = \D + U\U^*-\D =
\D + U\D U^* - UU^*\D = \D + U[\D ,U^*]$ where $U[\D ,U^*] \in \Omega^1_{\D}$, the differential algebra.
\end{proof}

\section{Acknowledgements}
Grateful thanks are due to the following people (alphabetical order).  John Barrett, Paolo Bertozzini,
Ali Chamseddine, Roberto Conti, Louis Crane, Ryszard Nest, Pedro Resende. Grazie mille to Paolo
Bertozzini, a great thinker and teacher.

\end{document}